\documentclass[a4paper,11pt,twoside,tbtags]{article} %leqno,左に番号の時
\usepackage{amsmath}\usepackage{amsthm}\usepackage{amssymb}
\usepackage{latexsym}\usepackage{epsfig}\usepackage{graphicx}
\theoremstyle{plain}
\newtheorem{thm}{Theorem}[section]  

\newtheorem{lem}[thm]{Lemma}
\newtheorem{prop}[thm]{Proposition}
\newtheorem{dfn}[thm]{Definition}
\newcommand\DN{\newcommand}

\DN\lref[1]{Lemma~\ref{#1}}
\DN\tref[1]{Theorem~\ref{#1}}
\DN\pref[1]{Proposition~\ref{#1}}
\DN\sref[1]{Section~\ref{#1}}
\DN\ssref[1]{Subsection~\ref{#1}}
\DN\dref[1]{Definition~\ref{#1}}
\DN\rref[1]{Remark~\ref{#1}} 
\DN\corref[1]{Corollary~\ref{#1}}
\DN\eref[1]{Example~\ref{#1}}
\numberwithin{equation}{section}
\newcounter{Const} 

\DN\Ct{\refstepcounter{Const}c_{\theConst}}
\numberwithin{Const}{section}
\DN\cref[1]{c_{\ref{#1}}}	%%before

\DN\R{\mathbb{R}}
\DN\N{\mathbb{N}}
\DN\Q{\mathbb{Q}}
\DN\C{\mathbb{C}}
\DN\Z{\mathbb{Z}}

\DN\map[3]{#1\!:\!#2\!\to\!#3}
\DN\ot{ \otimes } 
\DN\ts{ \times }

\DN\limi[1]{\lim_{#1\to\infty}} 	
\DN\limz[1]{\lim_{#1\to0}}
\DN\limsupi[1]{\limsup_{#1\to\infty}} 	
\DN\limsupz[1]{\limsup_{#1\to0}}
\DN\liminfi[1]{\liminf_{#1\to\infty}} 	
\DN\liminfz[1]{\liminf_{#1\to0}}

\DN\sumii[1]{\sum_{#1=1}^{\infty}}
\DN\sumi[1]{\sum_{#1=0}^{\infty}}

\DN\PD[2]{\frac{\partial#1}{\partial#2}}
\DN\half{\frac{1}{2}}
\DN\Rd{\R ^d}
\DN\Zd{\Z ^d}

\DN\bs{\bigskip}
\DN\ms{\smallskip}

\DN\bigcupii[1]{\bigcup_{#1=1}^{\infty}}

\begin{document}

%===\input{bernoulli-def.tex}

\DN\hatk{\hat{K}}
\DN\muk{\mu ^{\K}}

\DN\No{\N_{0}}

\DN\CX{\mathsf{Conf}(X)}
\DN\BC{\mathcal{B} }
\DN\BCX{\mathcal{B}(\CX) }

%\DN\m{\mathsf{m}}
\DN\m{\lambda}

\DN\PM{\mathcal{P}_{inv} }

\DN\PS{(\Omega, \mathcal{F},\mathbb{P})}
\DN\MS{(\Omega, \mathcal{F})}
\DN\MSS{(\Omega', \mathcal{F}')}
\DN\DS{(\Omega, \mathcal{F},\mathbb{P},\mathsf{S}_G)}
\DN\DSS{(\Omega', \mathcal{F}',\mathbb{P}',{\mathsf{S}'}_G)}
\DN\DSZd{(\Omega, \mathcal{F},\mathbb{P},\mathsf{S}_{\Zd})}
\DN\DSRd{(\Omega, \mathcal{F},\mathbb{P},\mathsf{S}_{\Rd})}
\DN\DSZdn{(\Omega_n , \mathcal{F}_n ,\mathbb{P}_n ,\mathsf{S}_{n,\Zd})}
\DN\DSSZd{(\Omega', \mathcal{F}',\mathbb{P}',{\mathsf{S}'}_{\Zd})}

\DN\A{A}
%連続DPP

\DN\T{\mathsf{T}}
\DN\TR{\T _{\X}}
\DN\TZ{\T _{\Zd}}

\DN\K{K}

%分割

\DN\PP{\mathcal{P} }
\DN\PPn{\mathcal{P}_{n} }
\DN\PPno{\mathcal{P}_{n,\mathbf{0}}}
\DN\PPnom{\mathcal{P}_{n,\mathbf{0}}^{(m)} }
\DN\PPnz{\mathcal{P}_{n,z} }
\DN\PPnzm{\mathcal{P}_{n,z}^{(m)} }
\DN\PPnzmz{\mathcal{P}_{n,z}^{(m_z)}}

\DN\QQ{\mathcal{Q} }
\DN\QQo{\mathcal{Q}_{\mathbf{0}}}
\DN\QQom{\mathcal{Q}_{\mathbf{0}}^{(m)}}
\DN\QQz{\mathcal{Q}_{z} }
\DN\QQzm{\mathcal{Q}_{z}^{(m)} }
\DN\QQzmz{\mathcal{Q}_{z}^{(m_z)} }

\DN\RRz{R_{\infty,z}}
\DN\RRN{R_{N}}
\DN\RRNzl{R_{N,(z,l)}}

\DN\RRR{\mathcal{R} }
\DN\RRRo{\RRR_{\infty,\mathbf{0}}}
\DN\RRRz{\RRR_{\infty,z}}
\DN\RRRzm{\RRR_{\infty,z}^{(m)}}
\DN\RRRzmz{\RRR_{\infty,z}^{(m_z)}}

\DN\RRRN{\RRR_{N}}
%\DN\RRRNo{\RRR_{N,\mathbf{0}}}
\DN\RRRNzl{\RRR_{N,(z,l)}}
\DN\RRRNzlm{\RRR_{N,(z,l)}^{(m)}}
\DN\RRRNzlmzl{\RRR_{N,(z,l)}^{(m_{z,l})}}

\DN\SSS{\mathcal{S} }
\DN\SSSN{\SSS_{N}}
\DN\SSSNz{\SSS_{N,z}}
\DN\SSSNo{\SSS_{N,\mathbf{0}}}
\DN\SSSNzl{\SSS_{N,(z,l)}}
\DN\SSSNol{\SSS_{N,(\mathbf{0},l)}}

\DN\dd{\mathsf{d}}
\DN\HH{H}
\DN\h{h}

%-底空間
\DN\X{\Rd}

\DN\mX{\lambda_{\X}}
%-核関数

%\DN\k{k}
%-配置空間
\DN\XX{\mathsf{Conf}(\X)}
\DN\xx{\xi}
%-点過程
\DN\mXX{\mu ^{K}}
\DN\PSX{(\XX ,\mXX)}
%-測度力学系
\DN\DSXR{(\XX ,\mXX ,\TR)}
\DN\DSX{(\XX ,\mXX ,\TZ)}
\DN\DSXa{(\XX, \mXX, \T _{a \Zd})}
\DN\DSXX{(\CX,\mXX,\T_{X})}
%\DN\T{\mathsf{T}}

%-底空間
\DN\Y{\Zd}
\DN\mY{\lambda_{\Y}}
%-核関数
%\DN\KY{K}
%\DN\kY{k}
%-配置空間
\DN\YY{\mathsf{Conf}(\Y)}
\DN\yy{\eta}
%-点過程
\DN\mYY{\mu_{\Pn } ^{K}}
\DN\PSY{(\YY ,\mYY)}
%-測度力学系
\DN\DSY{(\YY ,\mYY ,\TZ)}

%-ONB
\DN\ONB{\Phi}
\DN\onb{\phi }

%-底空間
\DN\ZN{\Zd \times \N}
\DN\mZ{\lambda_{\Z}}
%-核関数
%\DN\KZ{K}
%\DN\kZ{k}
%-配置空間
\DN\ZZ{\mathsf{Conf}(\ZN)}
\DN\zz{\zeta}
%-点過程
\DN\mZZ{\nu ^{K, \ONB}}
\DN\PSZ{(\ZZ ,\mZZ)}
%-測度力学系
\DN\DSZ{(\ZZ ,\mZZ ,\TZ)}
\DN\DSZn{(\ZZ ,\mZZ_n ,\TZ)}

%-点過程
\DN\mZZr{\nu _r ^{K, \ONB}}
\DN\PSZr{(\ZZ ,\mZZr)}

\DN\mZZro{\nuo _{r} ^{K,\ONB}}
\DN\mZZru{\nuu _{r} ^{K,\ONB}}

%-測度力学系
\DN\DSZr{(\ZZ ,\mZZr ,\TZ)}
\DN\DSZro{(\ZZ ,\mZZro ,\TZ)}
\DN\DSZru{(\ZZ ,\mZZru ,\TZ)}

%-底空間
%\DN\NN{\mathsf{N}}
\DN\NN{[N]}
\DN\W{\Zd \times \NN}
\DN\mW{\lambda_{\W}}
%-核関数
\DN\KK{K^{\ONB}}
%\DN\KK{\mathbb{K}^{\ONB}}
\DN\KKr{\KK _{r}}
%-配置空間
\DN\WW{ \{0,1\}^{\Zd \times [N]}}
%\DN\WW{\mathsf{Conf}(\W)}
\DN\ww{\omega}
%-点過程
\DN\mWW{\nu_{N} ^{K,\ONB }}
\DN\PSW{(\WW ,\mWW)}
%-測度力学系
\DN\DSW{(\WW ,\mWW ,\TZ)}
\DN\DSWN{(\WW ,\mWW ,\TZ)}

%-点過程
\DN\mWWr{\nu_{r,N} ^{K,\ONB}}
\DN\PSWr{(\WW ,\mWWr)}
%-測度力学系
\DN\DSWr{(\WW , \mWWr , \TZ)}

%-分割
\DN\IP{1^P}
\DN\IQ{1^Q}
\DN\IPQ{1^{P \cup Q}}

\DN\dbar{\bar{d}}

\DN\Pn{P_n}
\DN\Pno{P_{n,\mathbf{0}}}
\DN\Pnz{P_{n,z}}
\DN\projn{\Pi _{\Pn}}

\DN\QN{Q^N}
\DN\QQN{Q_{z,l}^{N}}
\DN\proj{\varpi }
\DN\projN{\proj _N}
\DN\projNN{\proj _N ^{\prime}}

%----------------------------------

\DN\s{\{0,1\}^{\Zd} }

\DN\alg{\mathcal{B} }
\DN\algN{\mathcal{B}_N }

\DN\Ku{\underline{K}}
\DN\Ko{\overline{K}}

\DN\KKu{\underline{K }^{\ONB }}
\DN\KKo{\overline{K }^{\ONB }}

\DN\ok{\overline{\K}}
\DN\uk{\underline{\K}}

\DN\on{\overline{\nu}}
\DN\un{\underline{\nu}}

\DN\nuo{\on }
\DN\nuu{\un }
\DN\mWWro{\nuo _{r,N} ^{K,\ONB}}
\DN\mWWru{\nuu _{r,N} ^{K,\ONB}}

\DN\ol{\overline{\lambda}}
\DN\ul{\underline{\lambda}}
\DN\wh{\hat{w}(t)}
\DN\whr{\hat{w}_r (t)}
\DN\kh{\hat{k}(t)}
\DN\khr{\hat{k}_r (t)}
\DN\kr{k_r (x)}
\DN\mm{\gamma _N}
\DN\mcN{\mc _N}
\DN\mc{\gamma }

\DN\CPL{\Gamma}
\DN\cpl{\gamma}

\DN\dg{\mathsf{d}}
\DN\rr{r_0}

\DN\KKros{}
\DN\KKrus{}

%===================

\title{
Isomorphisms between 
determinantal point processes with translation invariant kernels 
and 
Poisson point processes
}
\author{Shota \textsc{Osada}$ $ \\ {\small s-osada@math.kyushu-u.ac.jp}
}

\maketitle
\abstract{
We prove the Bernoulli property for determinantal point processes on $ \Rd $ with translation-invariant kernels. 
For the determinantal point processes on $ \Zd$ with translation-invariant kernels, the Bernoulli property was proved by Lyons and Steif \cite{RL-St.03} and Shirai and Takahashi \cite{s-t.aop}. 
As its continuum version, we prove an isomorphism between the translation-invariant determinantal point processes on $ \Rd$ with translation-invariant kernels and homogeneous Poisson point processes. 
For this purpose, we also prove the Bernoulli property for the tree representations of the determinantal point processes. 
}

%=============================

%\input{Bernoulli-text.tex}

%\input{Bernoulli-ref.tex}

%==============================

%====================================\input{bernoulli-arxive.tex}

\DN\proofbegin{\smallskip \noindent {\it Proof. }}
\DN\proofend{\qed \smallskip }

\section{Introduction and the main result}\label{s:1} 
%\indent

%setting
We consider an isomorphism problem of measure-preserving dynamical systems among translation-invariant point processes on $ \Rd$ such as 
homogeneous Poisson point processes and determinantal point processes with translation-invariant kernel functions. 

The homogeneous Poisson point process is the point process such that 
numbers of particles on disjoint subsets follow independent Poisson distributions. 
The homogeneous Poisson point process is parameterized using the intensity $ r > 0$. 
From the general theory of Ornstein and Weiss \cite{o-w.amenable}, 
homogeneous Poisson point processes are isomorphic to each other regardless of the value of $ r$. 

The determinantal point process is the point process such that the determinants of a kernel function give its correlation functions. 
The determinantal point process describes a repulsive particle system 
and appears in various objects such as 
uniform spanning trees, the zeros of a hyperbolic Gaussian analytic function with a Bergman kernel, and the eigenvalue distribution of random matrices. 

These two classes of point processes have different properties in correlations among particles. 
For example, determinantal point processes have negative associations \cite{RL.03}. 
The sine point process is a typical example of a translation-invariant determinantal point process, that has number rigidity \cite{Bu-S.17}. 
On the other hand, Poisson point processes do not have this property because the particles are regionally independent. 
Nevertheless, we prove they are isomorphic to each other. 

We start by recalling the isomorphism theory. 

An automorphism $ \mathsf{S}$ of a probability space $ \PS$ is a bi-measurable bijection such that $ \mathbb{P} \circ \mathsf{S}^{-1} = \mathbb{P} $. 
%By definition $ \mathsf{S}$ is a transformation of $ \PS$. 
Let $ \mathsf{S}_{G}=\{\mathsf{S}_g : g \in G \}$ be a group of automorphisms of $ \PS$ parametrized by a group $ G$. 
A measure-preserving dynamical system of $ G$-action is the quadruple $ \DS$. 
We call $ \DS $ the $ G$-action system for short.

Let $ \DS$ and $ \DSS$ be $ G$-action systems. 
A factor map is a measurable map $ \phi : \Omega \rightarrow \Omega'$ such that \begin{align*}& 
\mathbb{P}\circ \phi^{-1}=\mathbb{P}', 
\quad \phi \circ \mathsf{S}_g (x)= {\mathsf{S}'}_g \circ \phi (x) \text{ for each } g \in G \text{ and a.s. } x \in \Omega
.\end{align*} 
In this case, 
we call $ \DSS$ the $ \phi$-factor of $ \DS$ or simply a factor of $ \DS $. 
An isomorphism is a bi-measurable bijection $ \phi : \Omega \rightarrow \Omega'$ such that both $ \phi $ and $ \phi ^{-1}$ are factor maps. 
If there exists an isomorphism $ \phi: \Omega \rightarrow \Omega' $, 
then $ \DS$ and $ \DSS$ are said to be isomorphic. 

Let $ \DS$ be a $ G$-action system with a measurable map $ \phi$ from $ \MS$ to $ \MSS$. 
Then 
$ (\Omega' ,\mathcal{F}_{\phi}, \mathbb{P}_{\phi}, \mathsf{S}_{G}^{\phi})$ 
is a $ G$-action system. 
Here, $ (\Omega' ,\mathcal{F}_{\phi}, \mathbb{P}_{\phi}) $ is the completion of $ (\Omega^{\prime}, {\sigma[\phi]}, {\mathbb{P}\circ \phi ^{-1}}) $, and $ \mathsf{S}_{G}^{\phi} =\{ \phi \circ \mathsf{S}_{g} \circ \phi ^{-1} : g \in G \}$. 
We also call the $ G$-action system 
%\begin{align}\notag %\label{:}&
$(\Omega' , \mathcal{F}_{\phi}, \mathbb{P}_{\phi}, \mathsf{S}_{G}^{\phi} )$
%\end{align}
the $\phi$-factor of $ \DS$.

A typical system with discrete group action is a Bernoulli shift. 
A $ G$-action Bernoulli shift is a system formed from the direct product of a probability space over $ G $ and the canonical shift. 
Ornstein \cite{Orns.1970a,Orns.1970b} proved that the $ \Z $-action Bernoulli shifts with equal entropy are isomorphic to each other. 
We call a system $ \DS$ Bernoulli 
if $ \DS$ is isomorphic to a Bernoulli shift. 
Ornstein and Weiss \cite{o-w.amenable} extended the isomorphism theory to amenable group actions. 
As a consequence of the general theory, all the homogeneous Poisson point processes on $ \Rd $ are isomorphic to each other regardless of their intensity.

%PP
Let $ X $ be a locally compact Hausdorff space with countable basis. 
We denote by $ \CX$ the set of all nonnegative integer-valued Radon measures on $ X$. 
We equip $ \CX $ with the vague topology, under which $ \CX $ is a Polish space. 
We call a Borel probability measure $ \mu $ on $ \CX $ a point process on $ X$. We say $ \mu $ is simple when $ \mu (\xi(\{x\})\geq2)=0$ for each $ x \in X$. 

Let $ \mu$ be a point process on $ X$. 
Throughout this paper, we write the completion of $\mu$ by the same symbol. 
We also write $ (\CX, \mu , \T _G) $ as the $ G$-action system 
made of the completion of $ (\CX, \mathcal{B}(\CX),\mu )$ and 
a $ G$-action group of automorphisms $ \T _G$. 

%PPP
A homogeneous Poisson point process with intensity $ r >0$ is the point process on $ \Rd$ satisfying: \\
\thetag{1} $ \xi(A) $ has a Poisson distribution with mean $ r |A| $ for each $ A \in \mathcal{B}(\Rd ) $. 

\noindent \thetag{2}
$ \xi(A_1),\ldots,\xi(A_{k}) $ are independent for any disjoint subsets $ A_1,\ldots,A_k \in \mathcal{B}(\Rd ) $. 
Here, $ \xi(A) $ is the number of particles on $ A $ for $ \xi \in \CX$ and $ |A| $ is the Lebesgue measure of $A$.

%DPP
A determinantal point process $ \mu $ on $ X $ is a point process associated with a kernel function $ \K: X \times X \rightarrow \C $ and a Radon measure $ \m $ on $ X $, for which the $ n$-point correlation function with respect to $ \m $ is given by 
\begin{align}\label{:11a}&
\rho _n (x_1,\ldots,x_n)= \det [ K(x_i,x_j) ]_{i,j=1}^{n}
\end{align}
for each $ n \in \N $. 
See \dref{:dfn2} for the definition of the $ n$-point correlation function. 
We call $ \mu $ a $ (\K ,\m )$-determinantal point process. 
If the context is clear, we omit $ \m $ calling $ \mu$ a $ \K$-determinantal point process. 
Throughout this paper, we assume that $ \m$ is the Lebesgue measure if $ X=\Rd $. 

Now, we state the main theorem: 
\begin{thm} \label{l:11}
Let $ \hatk \in L^1 (\Rd)$ such that $ \hatk (t) \in [0,1]$ for a.e.\ $ t \in \Rd $. 
Let $ \mXX $ be a determinantal point process on $ \Rd$ with translation-invariant kernel $\K $ 
such that
\begin{align}\label{:11z}&
\K (x,y)=\int_{\Rd} \hatk (t) e^{2\pi i (x-y) \cdot t}dt
.\end{align}
Then $ \DSXR$ is isomorphic to a Poisson point process. 
Here, $ \T_{a} : \sum_{i \in \N}\delta_{x_{i}} \mapsto \sum _{i \in \N} \delta_{x_{i}+a} $ for $ a \in \Rd$ and $ \T_{\Rd}=\{\T_{a}: a \in \Rd \}$.
\end{thm}

\noindent

%ASSUMPTION FOR DPP
We remark that 
the assumption for $ \K $ in \tref{l:11} implies the following condition \thetag{1}--\thetag{4} with $ X=\Rd $ and the Lebesgue measure $ \m $. \\
\thetag{1} $ \K : X \times X \rightarrow \C$ is Hermitian symmetric. \\
\thetag{2} For each compact set $ A \subset X$, the integral operator $ \K $ on $ L^2 (A,\m)$ is of trace class. \\
\thetag{3} $ \mathsf{Spec} \K \subset [0,1]$. \\
\thetag{4} $ \K (x,y)= \K(x-y,0)$. \\
Under assumptions \thetag{1}--\thetag{3}, there exists a unique $ (\K , \m )$-determinantal point process $ \mu $ with the kernel function $ \K$ \cite{sosh,s-t.jfa}.

%Under assumptions \thetag{1}--\thetag{4}, 
The $ \K $-determinantal point process $ \mu $ is translation-invariant because its $ n$-correlation functions are translation-invariant. 

%related results: discrete case, application. 

For determinantal point processes on $ \Zd $ with translation-invariant kernel and the counting measure, 
Lyons and Steif \cite{RL-St.03} and Shirai and Takahashi \cite{s-t.aop} independently proved the Bernoulli property, 
the latter giving a sufficient condition of the weak Bernoulli property under the assumption $ \K : \Zd \times \Zd \rightarrow \C $ satisfying \thetag{1}, \thetag{2}, $\mathsf{Spec}(\K) \subset (0,1)$, and \thetag{4}. 
We recall that the weak Bernoulli property is stronger than the Bernoulli property. 
Lyons and Steif \cite{RL-St.03} proved the Bernoulli property for the case 
$ \K $ satisfying \thetag{1}--\thetag{4}. 
\tref{l:11} is its continuum version. 

%technical topic
One of the ideas in \cite{RL-St.03} is using the dbar distance, 
which is a metric on the set of $ \Zd$-action systems and the Bernoulli property is closed under the metric \cite{Ornstein-74,o-w.amenable,steif.91}. 
However, the dbar distance does not work for systems with infinite entropy because entropy is continuous with respect to the dbar distance. 
In general, a translation-invariant point process on $ \Rd$ has infinite entropy. 
Therefore, we cannot apply the dbar distance to our case. 
Therefore, we construct point processes on a discrete set 
that approximate the determinantal point process on $ \Rd$. 
We prove the Bernoulli property of the discrete point processes. 
In turn, we can prove the isomorphism of the determinantal point process on $ \Rd$ and the Poisson point process via the tree representation \cite{o-o.tail}.

%technical_remark
To prove \tref{l:11}, we apply the general theory given by Ornstein and Weiss \cite{o-w.amenable}. 
We quote them in the form applicable to the $ \Rd $- and $ \Zd $-actions. 
We also refer to \cite{Ornstein-74} for the $ \Z $- and $ \R $-actions, and \cite{steif.91} for the $ \Zd $-action.

%construction
The outline of this paper is as follows. 
In \sref{s:x}, we collect notions related to the Bernoulli property. 
In \sref{s:2}, we introduce the kernel functions that approximate the determinantal kernel $ \K $ in \tref{l:11} uniformly on any compact sets on $ \Rd$. 
In \sref{s:3}, we introduce the tree representations of the determinantal point processes on $ \Rd $. 
We combine these representations with the kernels introduced in \sref{s:2}. 
The tree representations are determinantal point processes on $ \Zd \times \N$ and are translation-invariant with respect to the first coordinate. 
In \sref{s:4}, we prove the Bernoulli property of the tree representation using the properties of the dbar distance introduces in \sref{s:x}. 
In \sref{s:5}, we prove \tref{l:11} using the Bernoulli property of the tree representations.

\section{Notions related to the Bernoulli property}\label{s:x}

In this section, we collect properties of point processes without determinantal structure and notions related to the Bernoulli property.

We first recall the notion of the monotone coupling. 
For $ \zeta ^i =\{ \zeta _{z}^i \}_{z \in \Zd} \in \{ 0,1\}^{\Zd}$ ($ i=1,2$), 
we write $ \zeta ^1 \leq \zeta^2 $ if $\zeta _z^1 \leq \zeta _z^2 $ for each $ z \in \Zd$. 
We equip $ \{0,1\}^{\Zd}$ with the product topology. 
We call a continuous function $f : \{ 0,1\}^{\Zd} \rightarrow \R $ a monotone function on $ \{ 0,1\}^{\Zd} $ if $ \zeta^1 \leq \zeta^2 \Rightarrow f(\zeta^1) \leq f(\zeta^2)$. 
Let $ \alg $ be the Borel $ \sigma$-field of $ \{0,1\}^{\Zd} $. 
For probability measures $ \mu $ and $ \nu $ on $ (\{ 0,1 \}^{\Zd}, \alg)$, 
we write $ \mu \leq \nu $ if for each monotone function $ f $,
\begin{align}&\notag 
%\label{:}&
\int _{\s }f d\mu \leq \int _{\s }f d\nu 
.\end{align}
Let $ \nu _1 $ and $ \nu _2$ be probability measures on $ \{0,1 \}^{\Zd}$. 
We say a probability measure $ \cpl $ on $ \{0,1 \}^{\Zd} \times \{ 0,1 \}^{\Zd}$ is a monotone coupling of $ \nu _1 $ and $ \nu _2$ if the following hold:\\ 
\thetag{1} $\cpl (A \times \s )=\nu _1 (A) \text{ for }A \in \alg$. \\
\thetag{2} $\cpl (\s \times A)=\nu _2 (A) \text{ for }A \in \alg$. \\
\thetag{3} $\cpl (\{(\zeta ^1 , \zeta ^2) \in \s \times \s ; \zeta ^1 \leq \zeta ^2 \})=1$. 

\begin{lem}[e.g.\,\cite{Liggett-85}]\label{l:42}
For probability measures $ \mu $ and $ \nu $ on $ \{ 0,1 \}^{\Zd}$, the following statements are equivalent: \\
\thetag{1} $ \mu \leq \nu $.\\ 
\thetag{2} There exists a monotone coupling of $ \mu$ and $ \nu$. 
\end{lem}

We naturally regard a simple point process $ \mu $ on $ \ZN $ as a probability measure on $ \{0,1\}^{\ZN }$, denoted by the same symbol $ \mu $. 
We write $ \mu \leq \nu $ for simple point processes $ \mu $ and $ \nu $ if the corresponding probability measures on $ \{0,1\}^{\ZN }$ satisfy $ \mu \leq \nu $. 
We introduce the notion of monotone coupling for simple point processes on $ \ZN $ from that of the corresponding probability measures on $ \{0,1\}^{\ZN }$ in an obvious fashion.

Fix $ N \in \N$. 
We set $ \NN =\{ 1,\ldots,N \}$. 
%factor
Let $ \QN = \{ \QQN : (z,l) \in \W \}$ be a partition of $ \ZN $ such that 

\begin{align}\label{:33h}
\QQN =
\begin{cases}
\{(z,l)\} & \text{ for } l \in [N-1]
\\
\{(z,m) \in \ZN ; m \geq l \} & \text{ for } l=N
\end{cases}
\end{align}
for each $ (z ,l) \in \W $. 
For $ \xi \in \ZZ $, we set 
\begin{align}\notag &%\label{:}&
\omega_{z,l}^{N}(\xi) =
1_{\{ \xi ( \QQN ) \geq 1 \} }
.\end{align}
Let $ \projN : \ZZ \rightarrow \WW $ denote the map
\begin{align}\label{:33j}&
\xi 
%= \sum _{i \in \N } \delta _{(z_i , l_i)} 
\mapsto 
%\omega = 
\{ \omega_ {z,l}^{N}(\xi) \}_{(z,l) \in \Zd \times \NN }
.
\end{align}
We denote the image measure $ \nu \circ \projN ^{-1} $ by $ \nu_N $ for a point process $ \nu$ on $ \ZN $.

\begin{prop}\label{l:x31}
Let $ \mu $ and $ \nu$ be simple point processes on $ \ZN$. 
Assume $ \mu \leq \nu$. 
Then $ \mu_N \leq \nu_N$. 
\end{prop}

\begin{proof}
By assumption and \lref{l:42}, there exists a monotone coupling $ \mc $ of $ \mu $ and $ \nu $. 
Let $ \mcN (\xi,\eta) = \mc \circ (\projN (\xi) , \projN (\eta) )^{ -1} $. 
Then for $ A \in \mathcal{B}(\WW) $, 
\begin{align}\notag %\label{:}
\mcN (A \times \WW ) 
=& 
\mc \bigl( \{ (\xi,\eta) ;( \projN (\xi),\projN (\eta)) \in A \times \WW \} \bigr)
\\\notag 
=&
\mc \bigl( \projN ^{-1} (A) \times \projN ^{-1}(\WW ) \bigr)
\\ \notag 
=&
\mc \bigl( \projN ^{-1} (A) \times \ZZ \bigr)
\\\notag 
=&
\mu \bigl( \projN ^{-1} (A) \bigr)
\\\notag 
=&
\mu _N \bigl( A \bigr)
.\end{align}
The fourth equation follows from the fact that $ \mc $ is a coupling of $ \mu $ and $ \nu $. 
Since the same is true for $ \WW \times A $, we get that 
\begin{align}&\notag 
\mcN (\WW \times A ) = \nu_N (A)
.\end{align}
Moreover, 
by $ \mcN (\xi,\eta) = \mc \circ (\projN (\xi) , \projN (\eta) )^{ -1} $
\begin{align}&\notag %\label{:}&
\mcN ( \{ (\zeta , \omega ) \in \WW \times \WW ; \zeta \leq \omega \} ) 
\\ \notag = & 
\mc ( \{ (\xi , \eta ) \in \ZZ \times \ZZ ; \projN (\xi) \leq \projN (\eta) \})
\\ \notag = & 
1
.\end{align} 
The last equation follows from the fact that $ \mc $ is a monotone coupling of $ \mu $ and $ \nu $. 
Hence $ \mcN $ is a monotone coupling of $ \mu_N $ and $ \nu_N $. 
From this and \lref{l:42}, we obtain the claim. 
\end{proof}

We recall the notion of finitely dependence, which is a sufficient condition of the Bernoulli property. 
\begin{dfn}\label{dfn:31}
Let $ \Omega $ be a countable set. \\
\thetag{1} A probability measure $ \nu $ on $ \Omega ^{\Zd} $ is called $ r$-dependent if, for each $ R , S \subset \Zd$, 
\begin{align}\notag %\label{:}&\label{:}&
\inf\{ \dg(z,w); z \in R , w \in S \} \geq r
\Rightarrow 
 \sigma[ \pi _R ] \text{ and } \sigma[ \pi _S ] \text{ are independent}
.\end{align}
Here, $ \dg (z,w)$ is the graph distance on $ \Zd $ and $ \pi _R : \Omega^{\Zd} \rightarrow \Omega^{R} $ is the projection given by $ \{ \omega _{z} \}_{z \in \Zd} \mapsto \{ \omega _{z}\}_{z \in R}$. \\
\smallskip
\thetag{2} $ \nu $ is called finitely dependent if $ \nu$ is $ r$-dependent for some $ r \in \N $. 
\end{dfn}

Let $ \PM(M)$ be the set of translation-invariant probability measures on $ [M]^{\Zd}$. 
For $ x , y \in \Zd$, define $ x < y $ if $ x_i < y_i $ for $ i=\min\{ j=1,\ldots,d ; x_j \neq y_j \}$. For $ P , Q \subset \Zd$, we set $P < Q$ if $ x<y$ for all $ x \in P $ and $ y \in Q $ . 
\begin{dfn}[Very weak Bernoulli]\label{dfn:41}
We call $ \nu \in \PM (M)$ very weak Bernoulli if for each $ \epsilon >0$, there is a rectangle $ R \subset \Zd$ such that if, for any finite set $ Q=\{ x_1,\ldots,x_m \} <R $, there exists an $ \mathcal{A} \subset \sigma[\pi_Q] $ satisfying \eqref{:44x} and \eqref{:44y}$ :$
\begin{align} \label{:44x}
 &\ \nu(\bigcup _{A \in \mathcal{A}}A )>1-\epsilon 
.\\\label{:44y}
& \inf _{\cpl \in \CPL (\nu|_R , \nu_A |_R )}
 \mathsf{E}^{\cpl }[\frac{1}{\# R}
\sum _{z \in R}1_{\{X_z \neq Y_z\}} ] < \epsilon \text{ for } A \in \mathcal{A} 
.\end{align} 
Here $ \nu_A $ denotes the conditional probability measure under $ A$, 
$ \nu | _{R}= \nu \circ \pi _{R} ^{-1}$, 
and $ \nu_A | _{R} = \nu_{A} \circ \pi _{R} ^{-1}$. 
Furthermore, $ \CPL (\nu|_R , \nu_A |_R )$ is the collection of the couplings of $\nu | _{R}$ and $ \nu_A |_R $, 
and $ ((X_z)_{z \in R},(Y_z)_{z \in R}) $ is a random variable corresponding to $ \gamma $. 
\end{dfn}

\begin{lem} \label{l:44}
If $ \nu \in \PM(M)$ is finitely dependent, then $ \nu $ is very weak Bernoulli.
\end{lem}
\begin{proof}
By definition, there exists an $ r_0 $ such that $ \nu$ is $ r_0$-dependent. 
For $ \epsilon >0 $, let $ R \subset \Zd$ be a rectangle such that ${(r_0 )^d}/{\# R} < \epsilon$. 
Let $ Q =\{z_1,\ldots,z_m \} \subset \Zd $ be a finite set such that $ Q<R $. 
We set 
\begin{align}\notag %\label{:}&
Q_{r_0} = \{ w \in \Zd ; \dg (z,w) \leq r_0 \text{ for some } z \in Q \}
.\end{align}
Then $ \# R \cap Q_{r_0} \leq (r_0)^d$. 
By $ r_0$-dependence, $ \sigma[\pi_{Q}]$ and $ \sigma[\pi_{Q_{r_0}^c}]$ are independent under $ \nu$. 
Hence for each $ A \in \sigma[\pi_{Q}]$, $ \nu = \nu_A $ on $ \sigma[\pi_{Q_{r_0}^{c}}]$. 
Let $ \cpl _{A}$ be the coupling of $ \nu \circ \pi _R ^{-1} $ and $ \nu _A \circ \pi _R ^{-1}$ such that 
$ X_z = Y_z $ for $ z \in R \cap Q_{r_0}^{c} $ and $ X_z \perp Y_z$ for $ z \in R \cap Q_{r_0}$. 
Then 
\begin{align}\notag 
\mathsf{E}^{\cpl_A }[\frac{1}{\# R}
\sum _{z \in R}1_{\{X_z \neq Y_z\}} ] \leq \frac{(r_0 )^d}{\# R} < \epsilon 
.\end{align}
This implies the claim. 
\end{proof}
The very weak Bernoulli property is equivalent to the Bernoulli property for elements of $ \PM(M)$: 
\begin{lem}[\cite{Ornstein-74,o-w.amenable,steif.91}] \label{l:45}
For $ \nu \in \PM(M)$, the following statements are equivalent:\\
\thetag{1} $\nu \text{ is very weak Bernoulli.}$\\ 
\thetag{2} $\nu \text{ is isomorphic to a Bernoulli shift.}$
\end{lem}

\begin{lem}[e.g.\,\cite{RL-St.03}]\label{l:x33}
If $ \nu \in \PM(M)$ is finitely dependent, then $ \nu $ is isomorphic to a Bernoulli shift. 
\end{lem}
\begin{proof}
From \lref{l:44} and \lref{l:45}, we obtain the claim. 
\end{proof}

We quote Theorem 5 in III.6 in \cite{o-w.amenable}: 
\begin{lem}[\cite{Ornstein-74,o-w.amenable}] \label{l:41}
Let $ \DSZd $ be an ergodic system. 
Let $\{ \mathcal{F}_n : n\in\N\}$ be an increasing sequence of $ \mathsf{S}_{\Zd}$-invariant sub $ \sigma$-fields. 
Let $ \bigvee _{n \in \N} \mathcal{F}_n$ be the completion of $ \sigma[ \bigcup _{n \in \N} \mathcal{F}_n ] $. 
Assume that $\{ \mathcal{F}_n : n\in\N\}$ satisfies \eqref{:41a1} and \eqref{:41a2}$ :$ 
\begin{align}\label{:41a1}
& \bigvee _{n \in \N} \mathcal{F}_n = \mathcal{F}. 
\\\label{:41a2}
& 
\mathcal{F}_n \text{-factor is isomorphic to a Bernoulli shift for each } n.
\end{align}
Then $ \DSZd$ is isomorphic to a Bernoulli shift.
\end{lem}

Let $ \mu$ and $ \nu \in \PM(M)$. 
Define $ \dbar : \PM(M) \times \PM(M) \rightarrow [0,1]$ by 
\begin{align} \label{:44w}&
\dbar (\mu ,\nu )= \inf _{ \cpl \in \CPL (\mu , \nu) } 
\cpl ( \{ (\zeta , \omega ) \in [M]^{\Zd} \times [M]^{\Zd } ; \zeta_0 \neq \omega_0 \} )
. 
\end{align}
Then $ \dbar $ gives a metric on $ \PM(M)$.
The Bernoulli property is closed under $ \dbar$: 
\begin{lem}[\cite{Ornstein-74,o-w.amenable,steif.91}] \label{l:48}
Let $ \nu $ and $ \{\nu_n : n \in \N\}$ be elements of $ \PM(M)$. 
Suppose that $ \lim_{n \rightarrow \infty} \dbar(\nu_n , \nu)=0$ 
and that 
each $ \nu_n$ is isomorphic to a Bernoulli shift. 
Then $ \nu$ is isomorphic to a Bernoulli shift. 
\end{lem}

\begin{prop}\label{l:x32}

Let $ (\ZZ , \nu, \T_{\Zd}) $ be ergodic. Let $ \nu$ be simple. 
Suppose that there exists a sequence $\{ \nu _r : r \in \N \} $ of point processes on $ \ZN $ such that
\begin{align}\label{l:x32a}
 & \nu_{r,N}\text{ is isomorphic to a Bernoulli shift for each } r \text{ and } N \in \N. \\\label{l:x32b}
 & \lim_{r \to \infty} \dbar(\nu_{r,N},\nu_N) = 0 \text{ for each }N \in \N.
\end{align}
Here, $ \nu_{r,N}=\nu_r \circ \varpi_N^{-1} $ and $ \nu_N= \nu \circ \varpi_N^{-1}$. 
Then, $ \nu$ is isomorphic to a Bernoulli shift. 
\end{prop}
\begin{proof}
Recall that $ \QN = \{ \QQN : (z,l) \in \W \}$ is a partition of $ \ZN $. 
Here, $ \QQN $ is defined in \eqref{:33h}. 
Then $ \QN $ becomes finer as $ N \rightarrow \infty$ 
and 
$ \bigvee _{N \in \N} \QN $ separates points of $ \ZN $ 
by construction. 
Here, $ \bigvee _{N \in \N}\QN$ is the refinement of partitions $ \{ \QN \}_{N \in \N} $. 
From this, we obtain that 
$\{ \sigma [ \projN ] \}_{N \in \N }$ is increasing 
and 
$ \bigvee _{N \in \N} \sigma [ \projN ] $ separates points of $ \ZZ$. 
Hence $\{ \sigma [ \projN ] \}_{N \in \N }$ satisfies \eqref{:41a1}. 

From the assumptions \eqref{l:x32a} and \eqref{l:x32b} and \lref{l:48}, 
$ \nu_N $ is isomorphic to a Bernoulli shift. 
Hence $\{ \sigma [ \projN ] \}_{N \in \N }$ satisfies \eqref{:41a2}. 

From the above and \lref{l:41}, the claim holds. 
\end{proof}

\section{Approximations of the determinantal kernel}\label{s:2}

In this section, we introduce three approximations of the kernel $ K $ introduced in \eqref{:11z}. 

For $ r >0 $, let $ w_r : \Rd \rightarrow \R $ be the product of the tent function such that 
\begin{align} \notag %\label{:}
w_r(x) =& \prod _{j=1}^{d} (1-|x_j|/r)1_{\{|x_j|<r\}} (x)
.\end{align}
We denote by $ \hat{w}_r $ its Fourier transform 
\begin{align}\notag %\label{:}
\whr 
= 
\int _{\Rd } w_r(x) e^{2 \pi i x \cdot t } dx 
= 
r ^{-d}
\prod _{j=1}^{d} \bigl( \frac{\sin{\pi r t_j }}{\pi t_j } \bigr)^{2}
.\end{align} 
Let $ \hatk \in L^1 (\Rd)$ such that $ \hatk (t) \in [0,1]$ for a.s. $ t \in \Rd $. 
Set $ \hatk _r (t)= \hatk \ast \whr$. 
Then by $ \| \hat{w}_r \|_{L^1 (\Rd)} =1$, $ \hatk _r (t) \in [0,1]$ for a.e. $ t \in \Rd $. 
Let 
\begin{align} \label{:21a} 
\uk_r (x,y) =& \int _{\Rd} 
\bigl( \hatk_r (t) \wedge \hatk (t) \bigr) e^{2 \pi i (x-y) \cdot t} dt
, 
\\ 
\label{:21b}
\K _r(x,y)=& \int _{\Rd} 
\hatk _r (t) e^{2 \pi i (x-y) \cdot t} dt 
, 
\\ \label{:21c}
\ok_r (x,y) =& \int _{\Rd} 
\bigl( \hatk_r (t) \vee \hatk (t) \bigr) e^{2 \pi i (x-y) \cdot t} dt 
. 
\end{align} 
Here, $ a \wedge b = \max\{a,b\}$ and $ a \vee b = \min\{a,b\}$ for $ a,b \in \R$, respectively.
Then $ \uk _r , \K _r $, and $ \ok _r $ satisfies \thetag{1}--\thetag{4} before \tref{l:11}. 

%For $ \K : X \times X \mapsto \C $ satisfying (\ref{:21z}), we define by the same symbol the integral operator $ \K f (x)= \int _{X} \K(x,y)f(y) \m(dy) $ for $ f \in L^2 (X ,\m)$. 
For $ \K : X \times X \mapsto \C $, we denote $ O \leq K $ if $ K $ is nonnegative definite as an integral operator on $ L^2 (\Rd)$
and 
$ \K _1 \leq \K _2$ if $ K_2 - K_1 $ is nonnegative definite. 
\begin{lem} \label{l:21}
Let $ \uk _r , \K _r $, and $ \ok _r $ be as \eqref{:21a}, \eqref{:21b}, and \eqref{:21c}, respectively. Then 
\begin{align}\label{:21f}&
\Ku _r \leq K \leq \Ko _r 
,\\& \label{:21g}
\Ku _r \leq K_r \leq \Ko _r 
. 
\end{align}
\end{lem}
\begin{proof}[Proof]
By construction, we see
\begin{align}\notag&
\hatk_r (t) \wedge \hatk (t) \leq \hatk (t) \leq \hatk_r (t) \vee \hatk (t)
\\& \notag
\hatk_r (t) \wedge \hatk (t) \leq \hatk_r (t) \leq \hatk_r (t) \vee \hatk (t)
. \end{align}
From \eqref{:21a}-\eqref{:21c} combined with above inequalities,
we obtain \eqref{:21f} and \eqref{:21g}. 
\end{proof}

\section{Tree representations of determinantal point processes} \label{s:3}

In this section, we introduce tree representations of determinantal point processes on $ \Rd$. 
Then we apply it to the determinantal point processes associated with the kernels introduced in \sref{s:2}. 
Before doing this, we recall the definition and well-known facts about determinantal point processes.

Let $ \mu$ be a point process on $ X$. 
A locally integrable symmetric function $ \rho^n : X^n \rightarrow [0,\infty) $ is called the $ n$-point correlation function of $ \mu $ (with respect to a Radon measure $ \m$ on $ X$ ) if 
\begin{align}\label{:dfn2}&
\mathsf{E}^{\mu}\bigl[\prod _{ i=1}^{k}\frac{\xx(A_i)!}{(\xx(A_i)-n_i)!} \bigr]=
\int _{A_1^{n_1}\times \cdots \times A_k^{n_k}}\rho^{n}(x_1, \ldots, x_n)
\m(dx_1)\cdots\m(dx_n)
\end{align} 
for any disjoint Borel subsets $ A_1,\ldots,A_k$ and for any $ n_i \in \N$, $ i=1,\ldots,k$ such that $ \sum_{i=1}^{k}n_i=n$. 
Let $ \K : X\times X \rightarrow \C $. 
We call $ \mu $ a determinantal point process 
with the kernel $ K $ and the Radon measure $ \m$ if 
the $ n$-point correlation function $ \rho ^n $ of $ \mu$ with respect to $ \m $ satisfies \eqref{:11a} for each $ n$. 

Assume $ \K : X \times X \rightarrow \C $ satisfies: 
\begin{align}\label{:31z1} 
 & \overline{ \K (x,y)}=\K (y,x).\\\label{:31z2} 
 & K_A \text{ is trace class for any compact } A \subset X. \\\label{:31z3} 
 & \mathsf{Spec}( \K) \subset [0,1].
\end{align}
Here, $ \K $ in \eqref{:31z3} is an integral operator on $ L^2(X,\m )$ such that 
$ \K f (x) = \int _X \K (x,y) \m (dy)$ and $ K_A$ in \eqref{:31z2} is its restriction on $ L^2 (A,\m)$. 
Then there exists the unique determinantal point process on $ X $ with the kernel function $ \K$. 
%Throughout this paper, let $ \m $ be the Lebesgue measure when $ X=\Rd $, and be the counting measure when $ X= \Zd$, respectively. 

Next, we introduce the tree representations of the determinantal point processes. 
Let $ \mXX $ be the determinantal point process on $ \Rd$ with the kernel function $ \K $ satisfying \eqref{:31z1}--\eqref{:31z3}. 
First, we introduce a partition of $ \Rd$ and the associated orthonormal basis on $ L^2 (\Rd)$. 
Let $ P =\{ P _{z} : z \in \Zd \}$ be a partition of $ \Rd $ such that each $ P_{z}$ is relatively compact and
\begin{align}&\notag 
%\label{:}&
P _{z+w}=P_{z}+w \text{ for } z,w \in \Zd
.\end{align}
Here, $ A + x =\{ a+x ; a \in A\} $ for $ A \subset \Rd$ and $ x \in \Rd$. 
Let $ \ONB = \ONB _{P}=\{ \onb_{z,l} \}_{(z,l) \in \ZN }$ be an orthonormal basis on $ L^2 (\Rd)$ such that $ \mathsf{supp}\onb _{z,l} \subset P_{z} $ and 
\begin{align}\label{:31a}&
\onb _{z+w,l}(x)=\onb _{z,l} (x-w)
.\end{align}
For the kernel function $ \K $ above, 
let 
\begin{align}\label{:31b}&
\KK (z,l;w,m)=\int _{\Rd \times \Rd} \onb _{z,l} (x) \K (x,y) \onb _{w,m}(y) dxdy
.\end{align}
Then $ \KK $ is a kernel function on $ \ZN $. 
$ \KK $ satisfy \eqref{:31z1}--\eqref{:31z3} with respect to the counting measure on $ \ZN $: 
\begin{lem} \label{l:4x}
Assume that $ \K $ satisfies \eqref{:31z1}--\eqref{:31z3} with respect to $ L^2(\Rd)$. Then $ \KK $ satisfies \eqref{:31z1}--\eqref{:31z3} with respect to the counting measure on $ \ZN $. 
\end{lem}
\begin{proof}
By assumption and \eqref{:31b}, $ \KK $ satisfies \eqref{:31z1} and \eqref{:31z2}. 
\eqref{:31z3} follows from Lemma 2 in p.430 of \cite{o-o.tail}. 
\end{proof}
From \lref{l:4x} and the general theory in \cite{sosh,s-t.jfa}, there exists a determinantal point process $ \mZZ $ on $ \ZN $ associated with $ \KK $. 
We call $ \mZZ $ the tree representation of $ \mXX$ with respect to $ \ONB$. 
\begin{lem}[\cite{o-o.tail}] \label{l:4x2}
Let $ \pi : \ZZ \rightarrow \mathsf{Conf}(\Zd)$ such that 
\begin{align}\notag %\label{:42c}&
\eta \mapsto \pi(\eta)=\sum_{z \in \Zd} \eta(\{ z \} \times \N )\delta_{z}
.\end{align}
Then for $ A \in \sigma[\{\xi \in \ZZ ; \xi(P_z)=n\};z \in \Zd , n \in \N] $, 
\begin{align}&\notag %\label{:}&
\mZZ \circ \pi ^{-1} (A) = \mXX (A)
.\end{align} 
\end{lem}
\begin{proof}
From Theorem 2 in p.427 of \cite{o-o.tail}, we easily obtain the claim. 
\end{proof}

We apply the tree representations for the translation-invariant kernels on $ \Rd $ introduces in \sref{s:2}.

Assume that $ \K $ is given by \eqref{:11z}. 
Then $ \K $ is translation-invariant. 
Hence by construction $ \KK $ is translation-invariant with respect to the first coordinate $ \Zd$. 
From this we see that $ \mZZ $ is translation-invariant with respect to the first coordinate. 

Define $ \KKu _r$, $ \KK _r $, and $ \KKo _r$ similarly as \eqref{:31b} with replacement of $ \K $ with $ \Ku _r $, $ \K _r$, and $ \Ko _r$ in \eqref{:21a}--\eqref{:21c}, respectively. 
By construction, $ \Ku _r $, $ \K _r$, and $ \Ko _r$ satisfies \eqref{:31z1}--\eqref{:31z3}. 
Hence $ \KKu _r$, $ \KK _r $, and $ \KKo _r$ 
satisfy \eqref{:31z1}--\eqref{:31z3} with respect to the counting measure on $ \ZN $ by \lref{l:4x2}. 
Furthermore, $ \KKu _r$, $ \KK _r $, and $ \KKo _r$ are translation-invariant with respect to the first coordinate $ \Zd$. 

Let $\mZZru $, $ \mZZr $, and $ \mZZro $ be $ \KKu _r$-, $ \KK _r$- and $ \KKo _r$-determinantal point process, respectively. 
We remark that a determinantal point process $ \nu$ on $ \Zd$ has no multiple points with probability $ 1$.
Hence we can regard $ \nu$ as a probability measure on $ \{ 0,1 \}^{\Zd}$. 
We quote: 
\begin{lem}[\cite{RL.03}] \label{l:23}
Let $ \K _i : \Zd \times \Zd \rightarrow \C $ satisfying \eqref{:31z1}--\eqref{:31z3} $ (i=1,2)$. 
Assume that $ \K _1 \leq \K _2 $. 
Let $ \nu ^{\K _1}$ and $ \nu ^{\K _2}$ be the determinantal point processes with $ \K _1$ and $ \K _2$, respectively. 
Then there exists a monotone coupling of $ \nu ^{\K _1}$ and $ \nu ^{\K _2}$. 
\end{lem}
Applying \lref{l:23}, we obtain the following: 
\begin{lem} \label{l:33}
Let $\mZZru $, $ \mZZ $, $ \mZZr $, and $ \mZZro $ be determinantal point processes on $ \ZN $ as above. 
Then 
\begin{align}\label{:33a}&
\mZZru \leq \mZZ \leq \mZZro, 
\\\label{:33b}&
\mZZru \leq \mZZr \leq \mZZro
.\end{align}
\end{lem}

\begin{proof}[Proof]
Recall that $ \ONB $ is the orthonormal basis of $ L^2(\Rd )$ given in \eqref{:31a}. 
Let $ \map{U}{L^2(\Rd )}{L^2(\ZN )}$ be the unitary operator such that 
$ U (\onb _{z,n}) = e_{z,n}$, where 
$ \{e_{z,n} \}_{(z,n)\in \ZN }$ 
is the canonical orthonormal basis of $L^2(\ZN )$. 
Then by Lemma 1 in Section 3 of \cite{o-o.tail} we see that $ \KK = U \K U^{-1}$. 
From this and \lref{l:21}, we obtain 
\begin{align} \label{:33e}
\KKu _r \leq \KK \leq \KKo _r ,
\\ \label{:33f}
\KKu _r \leq \KK _r \leq \KKo _r
.\end{align}
From \eqref{:33e} and \eqref{:33f} combined with \lref{l:23}, we conclude \eqref{:33a} and \eqref{:33b}. 
\end{proof}

Recall that $ \KKu _r$, $ \KK _r $, and $ \KKo _r$ are translation-invariant with respect to the first coordinate. 
Hence $\mZZru $, $ \mZZr $, and $ \mZZro $ are also translation-invariant with respect to the first coordinate. 
Let $ \T_{\Zd}=\{\T_a : a \in \Zd \}$ and 
\begin{align}\notag 
\text{$ \T_{a}: \sum _{i \in \N} \delta_{(z_i,l_i)} \mapsto \sum_{i \in \N} \delta _{(z_i + a,l_i)}$ for $ a \in \Zd$.}
\end{align}
Then $\DSZru $, $ \DSZ $, $ \DSZr $, and $ \DSZro $ are $ \Zd$-action systems.

\section{The Bernoulli property of $ \DSZ $}\label{s:4}
We continue the setting of \sref{s:3}. 
Let $ \KK $ be the kernel defined by \eqref{:31b}. 
Let $ \mZZ $ be the $ \KK $-determinantal point process as before. 
The purpose of this section is to prove the Bernoulli property for $ \DSZ$.

Let $ \projN $ be the map defined by \eqref{:33j}. 
Let $ \DSWr $ denote the $ \varpi _N$-factor of $ \DSZr $. 
Here, $ \TZ $ in $ \DSWr $ is the shift of $ \WW $ such that for each $ a \in \Zd$
\begin{align}\notag &
\T _a : \omega = \{\omega_{z,l} \}_{(z,l)\in \W } \mapsto \{\omega_{z + a ,l} \}_{(z,l)\in \W } 
.\end{align}
We also denote $ \varpi _N$-factors of $\DSZru $, $ \DSZr $, and $ \DSZro $ 
by $(\WW , \mWWru , \TZ )$, $ (\WW , \mWWr , \TZ )$, and $(\WW , \mWWro , \TZ )$, respectively. 
We shall prove that $ \DSW $ is isomorphic to a Bernoulli shift. 
\begin{lem}	\label{l:43}
\begin{align}&\notag 
\mWWru \leq \mWW \leq \mWWro 
,\\ &\notag 
\mWWru \leq \mWWr \leq \mWWro 
.\end{align}
\end{lem}
\begin{proof}
From \pref{l:x31} and \lref{l:33}, we obtain the claim. 
\end{proof}

\begin{lem} \label{l:46}
$ \DSWr $ is isomorphic to a Bernoulli shift.
\end{lem}
\begin{proof}[Proof] 
We identify $ \WW $ with $ [2^N]^{\Zd}$ and $ \mWWr$ with an element of $ \PM(2^{N})$, respectively. 
We shall prove that $ \mWWr $ is finitely dependent. 
For this it only remains to prove that $ \mZZr$ is finitely dependent 
because $ \DSWr $ is the $ \varpi_N $-factor of $ \DSZr$. 

Let $ \dg $ be the graph distance as before. 
Let $ \rr >0 $ such that for each $ z,w \in \Zd$, 
\begin{align}\notag %\label{:}& 
\dg (z,w) \geq \rr \Rightarrow \inf \{ |z_i - w_i| ; i=1,\ldots,d \}\geq r
.\end{align}
For $ P , Q \subset \ZN $, we define a pseudo distance by 
\begin{align}\notag %\label{:}&
\dg(P,Q) =\inf\{ \dg(z,w); (z,l) \in P , (w,m) \in Q \}
.\end{align}
Let $ P , Q \subset \ZN$ be finite sets such that $ \dg(P,Q) \geq \rr $. 
Then 
\begin{align}\label{:37aaa}&
\text{$ \KK _r (z,l;w,m)=0$\quad for $ (z,l) \in P $, $(w,m) \in Q$ }
.\end{align}
For $ P \subset \ZN $, we define a cylinder set by 
\begin{align}&\notag %\label{:}&
\IP = \{ \omega \in \mathsf{Conf}(\ZN) ; \omega (\{(z,l)\})=1 \text{ for all } (z,l) \in P \} 
.\end{align}
By construction, $ \IP \cap \IQ = \IPQ$. 
Therefore 
\begin{align}\notag 
\mZZr ( \IP \cap \IQ ) 
=& 
\mZZr ( 1^{P \cup Q} )
\\\notag 
=& 
\det [ \KK _r ( z,l ; w,m ) ] _{ (z,l) , (w,m) \in P \cup Q }
\\\notag 
=&\notag
\det [ \KK _r ( z,l ; w,m ) ]_{(z,l) , (w,m) \in P } 
\det [ \KK _r ( z,l ; w,m ) ]_{(z,l) , (w,m) \in Q } 
\\ \label{:37b}
=& \mZZr (\IP) \mZZr (\IQ) 
.\end{align}
The third equality follows from \eqref{:37aaa}. 

Let $ R , S \subset \Zd $ such that $ \dg (R\times \N ,S\times \N) \geq r_0$. 
From \eqref{:37b} and the $ \pi$-$ \lambda$ theorem, 
\begin{align}\notag %\label{:}&
\mZZr(A \cap B)= \mZZr(A)\mZZr(B)
\end{align}
for each $ A \in \sigma[\pi_{R\times \N } ]$ and $ B \in \sigma[\pi_{S \times R}]$. 
Hence $ \mWWr $ is $ \rr $-dependent. 

From this and \lref{l:x33}, the claim holds. 
\end{proof}

\begin{lem} \label{l:47}For each $ N $, 
\begin{align}\label{:47a}&
 \lim_{r \rightarrow \infty} \dbar ( \mWW , \mWWr ) = 0
.\end{align}
\end{lem}
\begin{proof}[Proof]
Because $ \dbar$ is a metric on $ \PM(M)$, 
\begin{align} \label{:47a1}&
\dbar(\mWW , \mWWr ) \leq \dbar( \mWWru , \mWW ) + \dbar(\mWWru , \mWWr )
, \\ \label{:47a2}&
\dbar(\mWW , \mWWr ) \leq \dbar(\mWW , \mWWro ) + \dbar( \mWWr ,\mWWro )
.\end{align}
From \lref{l:42} and \lref{l:43}, 
 there exist a monotone coupling $ \mm $ of $ \mWWru $ and $ \mWW $. 
By the definition \eqref{:44w} of $ \dbar$, we deduce 
\begin{align}\notag 
\dbar(\mWWru , \mWW ) \leq& 
\mm \bigl( \bigl\{ (\ww _1 , \ww _2) \ ; \ \ww _1 (\{0\} \times \{l\}) \neq \ww _2 (\{0\} \times \{l\}) \text{ for } ^{\exists}l \in \NN \bigr\} \bigr) 
\\ 
\leq&\notag 
\sum _{l \in \NN } \mm \bigl( \bigl\{ (\ww _1 , \ww _2) \ ; \ \ww _1 (\{0\} \times \{l\}) \neq \ww _2 (\{0\} \times \{l\}) \bigr\} \bigr) 
\\
=&\label{:47b} \sum _{l \in \NN } 
\bigr\{
\mWW (\ww _1 (\{0\} \times \{l\})=1) 
- 
\mWWru (\ww _2 (\{0\} \times \{l\})=1) 
\bigl\}
.\end{align}
The last equation follows from the fact that $ \mm $ is a monotone coupling of $ \mWWru $ and $ \mWW $. 
Because of \lref{l:43}, 
\eqref{:47b} is true for 
$ ( \mWWru , \mWWr ) $, 
$ (\mWW , \mWWro ) $, 
and 
$ (\mWWr , \mWWro ) $. 
From this combined with \eqref{:47a1} and \eqref{:47a2}, we obtain that 
\begin{align} \notag 
\dbar(\mWW , \mWWr ) \leq 
%\dbar( \mWWru ,\mWWro ) \leq& \mm \bigl( \bigl\{ (\ww _1 , \ww _2) \ ; \ \ww _1 (\{0\} \times \{l\}) \neq \ww _2 (\{0\} \times \{l\}) \text{ for } ^{\exists}l \in \NN \bigr\} \bigr) \\ \leq&\notag \sum _{l \in [m]} \mm \bigl( \bigl\{ (\ww _1 , \ww _2) \ ; \ \ww _1 (\{0\} \times \{l\}) \neq \ww _2 (\{0\} \times \{l\}) \bigr\} \bigr) \\ = 
\notag &\sum _{l \in \NN } 
\bigr\{ 
\mWWro (\ww _1 (\{0\} \times \{l\})=1) - 
\mWWru (\ww _2 (\{0\} \times \{l\})=1)
\bigl\}
\\= &\notag 
\sum _{l \in [N-1]} 
\bigr\{ 
\mZZro (\ww _1 (\{0\} \times \{l\})=1)- 
\mZZru (\ww _2 (\{0\} \times \{l\})=1)
\bigl\}
\\&\label{:47c}%henkou
+\mZZro (\ww _1 (\{0\} \times \N \backslash \NN ) \geq 1)
-\mZZru (\ww _2 (\{0\} \times \N \backslash \NN ) \geq 1).
\end{align}
The last equation follows from definitions of $\mZZro $ and $\mZZru $. 

For $ (z,l) $ and $ (w,m )$, 
\begin{align} & \label{:47d}
| \KKo _r (z,l;w,m) - \KKu _r (z,l;w,m) | 
\\ \notag 
& =
\Bigl| \int _{\Rd \times \Rd} \{ \Ko _r (x,y) - \Ku _r (x,y) \} \onb _{z,l}(x) \onb _{w,m}(y) dxdy \Bigr|
\\ \notag 
& \leq 
\int _{\Rd \times \Rd} |\Ko _r (x,y) - \Ku _r (x,y)| |\onb _{z,l}(x) \onb _{w,m}(y)|dxdy
\\ \label{:47e}
& =
\int _{\mathsf{supp}\onb _{z,l }\times \mathsf{supp}\onb _{w,m}} |\Ko _r (x,y) - \Ku _r (x,y)| |\onb _{z,l}(x) \onb _{w,m}(y)|dxdy
%\\& =\int _{\mathsf{supp}\onb _{(z,l) }\times \mathsf{supp}\onb _{(w,m)}} |\Ko _r (x,y) - \Ku _r (x,y)| |\onb _{(0,l)}( x - a \cdot z ) \onb _{(0,m)}( y - a \cdot w )|dxdy
.\end{align}
Because $ \Ko _r \, , \Ku _r \in L _{\mathrm{loc}} ^2 (\Rd \times \Rd) $ and $\onb _{z,l}$ and $ \onb _{w,m}$ are orthonormal basis on $ L^2 (\Rd)$ with relatively compact supports, the Schwarz inequality implies that 
\begin{align}\label{:47f}&
\eqref{:47e} 
\leq 
\Bigl( 
\int _{\mathsf{supp \onb _{z,l}} \times \mathsf{supp}\onb _{w,m}} 
| \Ko _r (x,y)- \Ku _r (x,y) | ^2 dxdy 
 \Bigr) ^{\frac{1}{2}}
\end{align}
Because $ \hatk_r \rightarrow \hatk \text{ in }L^1 (\Rd)$ as $ r \rightarrow \infty$, 
$ \Ku _r $ and $ \Ko _r$ converge to $ K$ uniformly on any compact set. 
Hence RHS of \eqref{:47f} goes to $ 0$ as $ r \rightarrow \infty$. 
This implies that \eqref{:47d} goes to $ 0$ as $ r \rightarrow \infty$.
Hence for each compact set $ R \subset \Zd \times \N$, 
\begin{align}&\notag 
\max \Bigl\{ | \KKo _r (z,l;w,m) - \KKu _r (z,l;w,m) | \ ; (z,l) , (w,m ) \in R \Bigr\}
\rightarrow
0 \text{ as }r \rightarrow \infty
. \end{align}
From this and Proposition 3.10 in \cite{s-t.jfa}, 
\begin{align}\label{:47g}&
\mZZro \ ,\mZZru \rightarrow \mZZ \text{ weakly as } r \rightarrow \infty
.\end{align}
Finally, \eqref{:47c} and \eqref{:47g} imply \eqref{:47a} .
\end{proof}

\begin{thm} \label{l:4X}
$ \DSZ$ is isomorphic to a Bernoulli shift.
\end{thm}
\begin{proof}
From \pref{l:x32}, \lref{l:46} and \lref{l:47}, the claim holds. 
\end{proof}

\section{Proof of \tref{l:11} } \label{s:5}
The purpose of this section is to complete the proof of \tref{l:11}.

We quote a general fact of isomorphism theory: 
\begin{lem}[\cite{Ornstein-74,o-w.amenable}] \label{l:51}
Let $ \DSSZd$ be a factor of $ \DSZd$.
If $ \DSZd$ is isomorphic to a Bernoulli shift, then $ \DSSZd$ is isomorphic to a Bernoulli shift.
\end{lem}

For $ n \in \N $, let $ \Pn =\{ \Pnz : z \in \Zd \}$ be a partition of $ \Rd $ such that 
\begin{align}&\notag 
\Pnz = \prod_{i=1}^{d} [\frac{ z_i}{2^{n-1}},\frac{ z_{i}+1 }{2^{n-1}}) \ , z=(z_1,\ldots,z_d) \in \Zd.
\end{align}
Let $ \Pi _{\Pn}: \mathsf{Conf}(\Rd) \rightarrow \mathsf{Conf}(\Zd)$ such that 
\begin{align}&\notag 
\xi \mapsto \sum_{z \in \Zd} \xi(\Pnz )\delta_{z}
. \end{align}
Then $ \Pi _{\Pn } \circ \T _z (\xi) = \T _z \circ \Pi _{\Pn } (\xi)$ 
for each $ z \in \Zd $ and $ \xi \in \XX$. 
Let $ \mYY = \mXX \circ \Pi _{\Pn}^{-1}$. 
Then $ \DSY$ is the $ \Pi _{\Pn}$-factor of $ \DSX$. 
\begin{lem} \label{l:52}
$ \DSY $ is isomorphic to a Bernoulli shift.
\end{lem}
\begin{proof}

Let $ \ONB _n =\{ \onb_{z,l}^{n} \}_{(z,l) \in \ZN }$ be an orthonormal basis on $ L^2 (\Rd)$ such that $ \onb _{z+w,l}^{n}(x)=\onb _{z,l}^{n} (x-w)$ and $ \mathsf{supp} \, \onb _{z,l}^n = P_{n,z} $. 
Let $ \mZZ $ be the tree representation of $ \mXX $ with respect to $ \ONB _n $. 
Let $ \pi : \ZZ \rightarrow \mathsf{Conf}(\Zd)$ such that 
\begin{align}\notag %\label{:42c}&
\eta \mapsto \pi(\eta)=\sum_{z \in \Zd} \eta( \{ z \} \times \N )\delta_{z}
.\end{align}
By construction, $ \pi \circ \T _z (\eta)= \T _z \circ \pi (\eta)$ 
for each $ z \in \Zd$ and $ \eta \in \ZZ$. 
From \lref{l:4x2}, 
\begin{align}&\notag %\label{:}&
\mZZ \circ \pi ^{-1} = \mYY
.\end{align} 
Hence $ \DSY $ is the $ \pi$-factor of $ \DSZ $. 
From \tref{l:4X}, $ \DSZ $ is isomorphic to a Bernoulli shift.
From this and \lref{l:51}, the claim holds.
\end{proof}

\begin{lem} \label{l:53}
$ \DSX$ is isomorphic to a Bernoulli shift. 
\end{lem}

\begin{proof}[Proof]
By construction, the sequence of partitions $ \{ \Pn : n \in \N \} $ is increasingly finer and separates the points of $ \Rd$.
From this, we obtain that 
$\{ \sigma [ \projn ] \}_{n \in \N }$ is increasing 
and 
$ \bigvee _{n \in \N} \sigma [ \projn ] $ separates the points of $ \XX$. 
Putting this, \lref{l:52} and \lref{l:41} together implies the claim. 
\end{proof}

We quote Theorem10 of III.6. in \cite{o-w.amenable}: 
\begin{lem}[\cite{o-w.amenable}]\label{l:54}
For an $\Rd $-action system $ \DSRd$, let $ \mathsf{S}_{\Zd}=\{ \mathsf{S}_{g}: g \in \Zd \}$ be the limitation on $ \Zd$-action of $ \mathsf{S}_{\Rd}$.
If $ \DSZd$ is isomorphic to a Bernoulli shift with infinite entropy, then $ \DSRd$ is isomorphic to a homogeneous Poisson point process. 
\end{lem}

We are now ready to complete the proof of \tref{l:11}.

\begin{proof}[{Proof of \tref{l:11} }]
From \lref{l:53}, $ \DSX$ is isomorphic to a Bernoulli shift. 
Because the restriction of $ \mXX $ on $ [0,1)^d$ is a non-atomic probability measure, the entropy of $ \DSX $ is infinite. 
Putting this and \lref{l:54} together implies the claim. 
\end{proof}

%==========================================

\noindent 
{\bf Acknowledgment: }\\
This work is supported by JSPS KAKENHI Grant Number 18J20465.

%=====\input{Bernoulli-ref.tex}

\end{document}